\documentclass[11pt,a4paper]{article}
\usepackage[T1]{fontenc}
\usepackage[utf8]{inputenc}
\usepackage{lmodern}
\usepackage[margin=24mm]{geometry}
\usepackage{amsmath,amssymb,amsthm}
\usepackage{microtype}
\usepackage{xcolor}

\usepackage[hidelinks]{hyperref}
\usepackage{xurl}
\newtheorem{theorem}{Theorem}
\newtheorem{lemma}[theorem]{Lemma}
\newcommand{\Pal}{\operatorname{Pal}}
\newcommand{\Z}{\mathbb{Z}}
\newcommand{\psemi}{\bar\delta^{0}}
\newcommand{\doi}[1]{\href{https://doi.org/#1}{\nolinkurl{#1}}}
\allowdisplaybreaks[1]
\title{Note on near-spanning balanced antidirected trees\\
missing from regular tournaments}
\author{Gregory Gutin\thanks{Royal Holloway University of London, UK.
Email: \href{mailto:g.gutin@rhul.ac.uk}{\texttt{g.gutin@rhul.ac.uk}}.}
\and
Yiming Hao\thanks{Nankai University, China.
Email: \href{mailto:1120230031@mail.nankai.edu.cn}{\texttt{1120230031@mail.nankai.edu.cn}}.}}
\date{}
\hypersetup{pdftitle={Note on near-spanning balanced antidirected trees missing from regular tournaments},
 pdfauthor={Gregory Gutin and Yiming Hao}}
\begin{document}
\maketitle

\begin{abstract}
We construct near-spanning balanced antidirected trees that fail to embed in highly symmetric
tournaments despite the conjectured semidegree condition being satisfied. More precisely, for every
sufficiently large odd integer $n$, we exhibit a regular vertex-transitive tournament on $n$ vertices and a
balanced antidirected caterpillar on $n-1$ vertices with maximum degree at most
$(1+o(1))n/\log_2 n$ that is not contained in the tournament. The host satisfies the strict $k/2$
threshold in both the semidegree and pseudo-semidegree settings, where $k$ is the number of arcs of the
target tree. Consequently, this gives counterexamples to the uniform sublinear-degree interpretations of
Conjectures~6.8 and~7.6 in Stein's survey (2024). The order $n-1$ of the target is best possible under the
strict semidegree hypothesis. We also give a six-vertex balanced antidirected double-star missing from
the seven-vertex Paley tournament, and show that seven is the smallest host order for a
pseudo-semidegree counterexample with this fixed double-star.
\end{abstract}

\section{Introduction}\label{sec:introduction}

An \emph{oriented graph} is a digraph with no loops and at most one arc between any two distinct
vertices. An oriented tree is \emph{antidirected} if every vertex is either a \emph{source} or a
\emph{sink}, and it is \emph{balanced} if its source and sink classes have the same size. We write
$\Delta(T)$ for the \emph{maximum degree} of the underlying tree. For an oriented graph $D$, let
\[
\delta^0(D)=\min_{v\in V(D)}\min\{d_D^+(v),d_D^-(v)\}
\]
be its \emph{minimum semidegree}, and let $\psemi(D)$ denote the minimum of its positive in- and
outdegrees, its \emph{minimum pseudo-semidegree}. We set $\psemi(D)=0$ when $D$ has no arcs.
\emph{Containment} always means an injective arc-preserving \emph{embedding}; copies need not be
induced.

Stein and Z\'arate-Guer\'en~\cite[Section~1, before Theorem~1.2]{SteinZarate}
already express the expectation that every balanced antidirected $k$-arc tree should occur in
an oriented graph of minimum semidegree greater than $k/2$, without a maximum-degree
restriction in that informal formulation.

Stein~\cite[Conjectures~6.8 and~7.6]{SteinSurvey} proposes the threshold $k/2$ for embedding
balanced antidirected trees. In the semidegree version, an $n$-vertex oriented graph $D$ satisfying
\[
\delta^0(D)>\frac{k}{2}
\]
is conjectured to contain every balanced antidirected $k$-arc tree $T$ with $\Delta(T)=o(n)$; the
pseudo-semidegree version replaces $\delta^0(D)$ by $\psemi(D)$. We interpret the condition
$\Delta(T)=o(n)$ uniformly, allowing a sequence of targets $T=T_n$ with
$\Delta(T_n)/n\to0$. Thus, for every sequence of pairs $(D_n,T_n)$ satisfying the corresponding
strict degree condition and this limit, the assertion would require $T_n\subseteq D_n$ for all
sufficiently large $n$.

Equivalently, the \emph{uniform sublinear-degree assertion} considered here is the following:
\begin{quote}
For every function $h:\mathbb N\to\mathbb R_{\ge0}$ with $h(n)=o(n)$, there is an integer
$n_0(h)$ such that, for every $n\ge n_0(h)$, every $n$-vertex oriented graph $D$ contains every
balanced antidirected tree $T$ satisfying
\[
\Delta(T)\le h(n),\qquad \delta^0(D)>\frac{|E(T)|}{2}.
\]
\end{quote}
The corresponding pseudo-semidegree assertion replaces $\delta^0(D)$ by $\psemi(D)$.
This quantification is stronger than an assertion for one prescribed degree bound, such as
$\Delta(T)\le(\log n)^c$ with fixed $c>0$, or the existence of some sufficiently slowly
growing admissible degree bound. Our counterexamples address the uniform assertions,
not all such weaker formulations.

Section~\ref{sec:seven} introduces a seven-vertex obstruction. Taking disjoint unions of the corresponding
host gives arbitrarily large oriented graphs avoiding the same fixed balanced antidirected tree, with
target maximum degree $o(n)$. These examples are not tournaments once more than one component is
used, and the target has fixed size. Theorem~\ref{thm:main} removes both limitations.

\begin{theorem}\label{thm:main}
For every sufficiently large odd integer $n$, there exist a regular vertex-transitive tournament $D_n$
on $n$ vertices and a balanced antidirected caterpillar $T_n$ such that
\begin{equation}\label{eq:main}
|V(T_n)|=n-1,\qquad
\Delta(T_n)\le(1+o(1))\frac{n}{\log_2 n},\qquad T_n\nsubseteq D_n.
\end{equation}
Moreover,
\begin{equation}\label{eq:main-semi}
\delta^0(D_n)=\psemi(D_n)=\frac{n-1}{2}
=\frac{|E(T_n)|}{2}+\frac12.
\end{equation}
The tree $T_n$ may be chosen as a function of $n$ alone.
\end{theorem}

The target order in Theorem~\ref{thm:main} is best possible under the strict semidegree hypothesis.
Indeed, every tournament on $n$ vertices has minimum semidegree at most $(n-1)/2$. Hence
$\delta^0(D)>k/2$ forces $k\le n-2$, so a $k$-arc tree has at most $n-1$ vertices. Our
construction attains this bound while keeping $\Delta(T_n)=o(n)$. The host also has the largest
possible minimum semidegree for an $n$-vertex tournament.

The host tournaments come from a result of Bonato, Gordinowicz and Pra{\l}at~\cite[Theorem~2.1]{BGP}.
A tournament $D$ is \emph{$s$-existentially closed} (\emph{$s$-e.c.}) if, whenever
$A,B\subseteq V(D)$ are disjoint and $|A|+|B|\le s$, there is a vertex $z\notin A\cup B$ such
that
\[
za\in E(D)\quad(a\in A),\qquad bz\in E(D)\quad(b\in B).
\]
We use the following consequence of their construction.

\begin{theorem}[Bonato--Gordinowicz--Pra{\l}at {\cite[Theorem~2.1]{BGP}}]\label{thm:ec}
For every sufficiently large odd integer $n$, there exists a regular vertex-transitive tournament on
$n$ vertices that is $s_n$-e.c., where

\[
M=\frac{n-1}{2},\qquad
s_n=\left\lfloor\log_2 M-4\log_2\log M-C\right\rfloor
\]
for a sufficiently large fixed absolute constant $C>0$. In particular,

\[
s_n=\log_2 n-O(\log\log n).
\]
\end{theorem}

To obtain this formulation, use the random circulant construction
of Bonato, Gordinowicz and Pra{\l}at with $p=1/2$.
On $\Z_{2M+1}$, form a \emph{connection set} $J$ by choosing independently and uniformly
one element of each pair $\{j,-j\}$, $1\le j\le M$, and define $D$ by $xy\in E(D)$ exactly when
$x-y\in J$. The resulting \emph{circulant tournament} has indegree and outdegree $M$
at every vertex, and translations act transitively as automorphisms. Thus it is regular
and vertex-transitive. By~\cite[Theorem~2.1]{BGP}, it has the $s_n$-extension property
with probability tending to one. Hence a suitable host exists for every sufficiently large
odd $n$. Fix $C$ and this integer sequence $s_n$ once and for all.

The cited theorem prescribes orientations on sets of total size exactly $s_n$.
For sufficiently large $n$ we have $s_n<n$; a smaller prescription can be enlarged to
size $s_n$ by assigning arbitrary orientations to additional vertices. A witness for the
enlarged prescription also satisfies the original one. This gives our convention
$|A|+|B|\le s_n$.

The proof of Theorem~\ref{thm:main} combines this extension property with a simple leaf obstruction.
We choose an antidirected spine of order about $\log_2 n$ and distribute the remaining vertices as
evenly as possible as leaves on the two sides of the spine. This keeps the maximum degree of order
$n/\log n$. Once the spine has been embedded, existential closure produces a host vertex whose
orientation towards the spine is incompatible with every possible leaf position. Since the target has
$n-1$ vertices, the single unused host vertex can be included among the prescribed adjacencies,
forcing the incompatible vertex to lie inside any putative copy.

Here and below, asymptotics are taken along odd integers $n\to\infty$; $\log$ denotes the
\emph{natural logarithm} and $\log_2$ the \emph{binary logarithm}. A \emph{caterpillar} is a tree
obtained by attaching leaves to a path, called its \emph{spine}.

The leaf obstruction has a classical undirected antecedent. The random-host example
of Koml\'os, S\'ark\"ozy and Szemer\'edi~\cite{KSS}, recalled
in~\cite[Section~6.3, after Theorem~6.10]{SteinSurvey}, uses a tree formed from a small
set of centres with many leaves, joined to form a tree, and a dense random host with no
suitable small dominating set. Replacing each undirected adjacency by arcs in both directions
does not give an oriented graph, so this is not a counterexample under our hypotheses.
Here the extension property yields a balanced antidirected obstruction in a regular
vertex-transitive tournament with only one unused host vertex.

A different earlier obstruction is given by Ara\'ujo, Santos and
Stein~\cite[Section~1, before Question~1.2]{ASS} and attributed there to Rozho\v{n}
and \v{S}\'amal: suitable powers of directed cycles omit certain spanning depth-two
out-trees of maximum degree $O(\sqrt n)$. Those trees are not antidirected, since their
intermediate vertices have both positive indegree and positive outdegree. They therefore
do not refute the balanced-antidirected assertions considered here. We are not aware of
earlier counterexamples to the precise uniform forms of Conjectures~6.8 and~7.6
specified above.

For caterpillars, Stein and Trujillo-Negrete~\cite[Proposition~1.3]{SteinTrujillo}
prove an embedding theorem under the different hypothesis $|E(D)|>(k-1)n$ for an
$n$-vertex digraph and a $k$-arc antidirected caterpillar. This density condition does
not hold for our hosts: for $k=n-2$ and $n>5$,
\[
|E(D_n)|=\frac{n(n-1)}2<n(n-3)=(k-1)n.
\]
Thus that result does not guarantee an embedding in the hosts of Theorem~\ref{thm:main}.

\section{A seven-vertex obstruction}\label{sec:seven}

Let $\Pal_7$ be the \emph{Paley tournament} on $\Z_7$, with
\begin{equation}\label{eq:paley}
xy\in E(\Pal_7)\quad\Longleftrightarrow\quad y-x\in\{1,2,4\}\pmod7.
\end{equation}

This is a classical obstruction for antidirected Hamilton paths; see
Gr\"unbaum~\cite{Grunbaum} and~\cite[Section~3.1]{SteinSurvey}.
Such a path on seven vertices has source and sink classes of sizes four and three and has
six arcs. It is therefore unbalanced, and $\delta^0(\Pal_7)=3=6/2$ gives only equality,
not the strict half-arc threshold. That earlier Hamilton-path obstruction is not itself
a counterexample to the assertions above. Our use of the same host instead has a balanced
six-vertex target and satisfies the strict inequality.

Let $S$ be the \emph{antidirected double-star} on
$\{a,b,u_1,u_2,v_1,v_2\}$ with
\begin{equation}\label{eq:double-star}
E(S)=\{ab,\ au_1,\ au_2,\ v_1b,\ v_2b\}.
\end{equation}
The source and sink classes of $S$ both have size three, so $S$ is balanced and antidirected.
Moreover,
\[
|E(S)|=5,\qquad\Delta(S)=3.
\]

We claim that $S\nsubseteq\Pal_7$. Suppose that $\varphi$ is an embedding, and put
$\alpha=\varphi(a)$ and $\beta=\varphi(b)$. Since $\alpha\beta\in E(\Pal_7)$, the difference
$\beta-\alpha$ belongs to $\{1,2,4\}$. This set is a multiplicative subgroup of $\Z_7^\times$, so
the map
\[
x\longmapsto(\beta-\alpha)^{-1}(x-\alpha)
\]
is an automorphism of $\Pal_7$. Composing $\varphi$ with this automorphism, we may assume that
$\varphi(a)=0$ and $\varphi(b)=1$. But
\[
N^+_{\Pal_7}(0)\setminus\{1\}=\{2,4\},\qquad
N^-_{\Pal_7}(1)\setminus\{0\}=\{4,6\}.
\]
Thus $\{\varphi(u_1),\varphi(u_2)\}=\{2,4\}$ and
$\{\varphi(v_1),\varphi(v_2)\}=\{4,6\}$, contradicting injectivity at the vertex~$4$. Since
$\Pal_7$ is regular,
\begin{equation}\label{eq:paley-semi}
\delta^0(\Pal_7)=\psemi(\Pal_7)=3>\frac{|E(S)|}{2}=\frac52.
\end{equation}
Thus this seven-vertex example already shows that the strict $k/2$ condition alone need not force the
containment of a balanced antidirected tree.

For this fixed target $S$, seven is also the smallest possible order of an oriented graph $D$
satisfying $\psemi(D)>5/2$ and $S\nsubseteq D$. To see this, suppose that $|V(D)|\le6$ and
$\psemi(D)>5/2$. Every positive in- or outdegree is then at least three. No vertex can have both
positive indegree and positive outdegree, since this would give
\[
6\le d_D^+(v)+d_D^-(v)\le|V(D)|-1\le5.
\]
Consequently, every nonisolated vertex is either a source or a sink. There is at least one arc, say
$xy\in E(D)$. Its endpoints satisfy $d_D^+(x)\ge3$ and $d_D^-(y)\ge3$, so there are at least three
sinks and at least three sources. Hence $D$ has exactly six vertices, with three sources and three
sinks, and every source sends an arc to every sink. Choosing any arc as the central arc and the other
four vertices as leaves gives a copy of $S$.

For $t\ge1$, let ${H_t}$ be the disjoint union of $t$ copies of $\Pal_7$. Because $S$ is connected,
any copy of $S$ in ${H_t}$ would have to lie inside a single component, and hence $S\nsubseteq {H_t}$.
At the same time,
\begin{equation}\label{eq:unions}
\delta^0({H_t})=\psemi({H_t})=3>\frac52,\qquad
\frac{\Delta(S)}{|V({H_t})|}=\frac3{7t}\longrightarrow0.
\end{equation}
Consequently, the seven-vertex example already yields arbitrarily large oriented hosts for which the
target has maximum degree $o(|V({H_t})|)$. Two limitations remain: the tree $S$ has the fixed number
$k=5$ of arcs, and ${H_t}$ is not a tournament when $t>1$. The main theorem removes both
restrictions.

In particular, a bounded maximum-degree assumption alone cannot repair a statement
that still allows a fixed target in arbitrarily large oriented hosts: here $k=5$ and
$\Delta(S)=3$ throughout. This observation does not address variants that also require
the target size $k$ to tend to infinity.

\section{Proof of the main theorem}\label{sec:proof}

We first establish a consequence of existential closure for a tree whose vertices outside a small set are
all leaves.

\begin{lemma}[Two-sided leaf obstruction]\label{lem:leaf}
Let $D$ be an $s$-e.c. tournament on $n$ vertices, and let $T$ be an antidirected tree on $n-1$
vertices. Suppose that $U\subseteq V(T)$ is nonempty and that every vertex outside $U$ is a leaf
whose neighbour belongs to $U$. If $|U|+1\le s$, then $T\nsubseteq D$.
\end{lemma}

\begin{proof}
Suppose that $\varphi$ embeds $T$ in $D$, and let $w$ be the unique vertex outside
$\varphi(V(T))$. Write $U^+$ and $U^-$ for the sources and sinks of $T$ that belong to $U$,
and put $X=\varphi(U^+)$ and $Y=\varphi(U^-)$. Since $|X|+|Y|+1\le s$, existential closure
applied to $X\cup\{w\}$ and $Y$ gives
\begin{equation}\label{eq:incompatible}
\begin{gathered}
z\notin X\cup Y\cup\{w\},\\
zx\in E(D)\quad(x\in X\cup\{w\}),\qquad
 yz\in E(D)\quad(y\in Y).
\end{gathered}
\end{equation}
Including $w$ among the prescribed vertices ensures that $z$ lies in the copy. Thus
$z=\varphi(v)$ for some $v\in V(T)\setminus U$, and $v$ is a leaf with a neighbour $u\in U$.
If $u\in U^+$, the arc $uv\in E(T)$ requires $\varphi(u)z\in E(D)$, contrary to
$z\varphi(u)\in E(D)$. If $u\in U^-$, the arc $vu\in E(T)$ requires
$z\varphi(u)\in E(D)$, contrary to $\varphi(u)z\in E(D)$. Either case
gives a contradiction.
\end{proof}

\begin{proof}[Proof of Theorem~\ref{thm:main}]
{Let $s_n$ be the integer sequence fixed in Theorem~\ref{thm:ec}; recall that
$s_n=\log_2 n-O(\log\log n)$.} For each
sufficiently large odd $n$, choose an $s_n$-e.c. regular vertex-transitive tournament $D_n$
supplied by that theorem. To apply Lemma~\ref{lem:leaf}, we need a spine with at most $s_n-1$
vertices. We take it as long as possible subject to having equally many sources and sinks, and set
\begin{equation}\label{eq:parameters}
m=\frac{n-1}{2},\qquad r=\left\lfloor\frac{s_n-1}{2}\right\rfloor.
\end{equation}
For large $n$ we have $2\le r\le m$, $2r+1\le s_n$, and
\begin{equation}\label{eq:r-asymptotic}
r=\frac12\log_2 n-O(\log\log n).
\end{equation}
Start with the path on the ordered vertices $a_1,b_1,a_2,b_2,\ldots,a_r,b_r$, with arc set
\begin{equation}\label{eq:spine}
\{a_i b_i:1\le i\le r\}\cup\{a_{i+1}b_i:1\le i<r\}.
\end{equation}
Attach $m-r$ new leaves to the vertices $a_1,\ldots,a_r$, with arcs directed away from the
spine, and another $m-r$ new leaves to $b_1,\ldots,b_r$, with arcs directed towards the spine.
In each group distribute the leaves as evenly as possible: if $m-r=qr+\ell$, where
$0\le\ell<r$, give the first $\ell$ vertices $q+1$ leaves each and every other vertex $q$ leaves.
Call the resulting tree $T_n$. Its underlying graph is a caterpillar, and every vertex is a source or
a sink. The source class consists of the $r$ vertices $a_i$ and the $m-r$ leaves attached to the
$b_i$'s; the sink class consists of the other $m$ vertices. Thus $T_n$ is balanced, with
\begin{equation}\label{eq:tree-order}
|V(T_n)|=2m=n-1,\qquad |E(T_n)|=2m-1=n-2.
\end{equation}
Each spine vertex has at most two neighbours on the spine and at most
$\lceil(m-r)/r\rceil$ attached leaves. All other vertices have degree one, so
\begin{equation}\label{eq:tree-degree}
\Delta(T_n)\le2+\left\lceil\frac{m-r}{r}\right\rceil
=1+\left\lceil\frac mr\right\rceil
=(1+o(1))\frac n{\log_2 n},
\end{equation}
where the last equality follows from $m\sim n/2$ and~\eqref{eq:r-asymptotic}.

Now take $U=\{a_1,b_1,\ldots,a_r,b_r\}$. Every vertex outside $U$ is a leaf adjacent to $U$,
and $|U|+1=2r+1\le s_n$. Lemma~\ref{lem:leaf} therefore gives $T_n\nsubseteq D_n$. Since
$D_n$ is regular,
\begin{equation}\label{eq:final-semi}
\delta^0(D_n)=\psemi(D_n)=m=\frac{|E(T_n)|}{2}+\frac12.
\end{equation}
The attachment rule determines $T_n$ from $n$ and the fixed sequence $s_n$, independently of the
choice of $D_n$. This proves the theorem.
\end{proof}

In particular, these pairs have $k_n=|E(T_n)|=n-2$ and $\Delta(T_n)/n\to0$, but
$T_n\nsubseteq D_n$. They therefore contradict the uniform sublinear-degree interpretations of
Conjectures~6.8 and~7.6 stated in the introduction. Since $k_n/n\to1$, the obstruction remains
after requiring $k_n\ge\eta n$ for any fixed $\eta\in(0,1)$; it is not a consequence of allowing
a fixed target in a growing disjoint union.

Moreover,
\[
\frac{\Delta(T_n)}{|E(T_n)|}
=\frac{\Delta(T_n)}n\,\frac{n}{n-2}\longrightarrow0.
\]
Consequently, replacing the condition $\Delta(T)=o(n)$ by $\Delta(T)=o(k)$ alone does
not avoid these near-spanning counterexamples.

Our construction does not address the corresponding assertions for growing targets of
bounded or polylogarithmic maximum degree in their full generality. In the near-spanning
regular-tournament setting, however, bounded maximum degree has a positive answer by
the spanning-tree theorem of Ara\'ujo, Santos and Stein~\cite[Theorem~1.1]{ASS}.
They prove that, for fixed $d\in\mathbb N$ and $\gamma>0$, every sufficiently large
$n$-vertex oriented graph $D$ with
\[
\delta^0(D)\ge(3/8+\gamma)n
\]
contains every oriented $n$-vertex tree of maximum degree at most $d$.

To deduce the near-spanning case, fix $d$ and let $T$ be any oriented tree on $n-1$
vertices with $\Delta(T)\le d$. Attach one new leaf to a leaf of $T$, with either
orientation, to form an $n$-vertex tree $T^+$ with
$\Delta(T^+)\le\max\{d,2\}$. For a regular tournament $D$ of odd order $n\ge9$,
\[
\delta^0(D)=\frac{n-1}{2}\ge\frac{7n}{16}
=\left(\frac38+\frac1{16}\right)n.
\]
The cited theorem with $\gamma=1/16$ therefore embeds $T^+$, and hence $T$, for all
sufficiently large odd $n$. The added leaf need not preserve balance, because that
theorem permits every orientation of the target. This deduction concerns fixed $d$;
it gives no such conclusion here for a prescribed growing polylogarithmic degree bound.

Nor do we claim that $n/\log n$ is the optimal maximum-degree scale: the sharp
feature established here is the one-vertex order deficit, not a universal degree threshold.

\section*{Statement on AI-assisted research}
GPT 6 Astra was used for literature retrieval, proof verification, and refining the manuscript's presentation.


\begin{thebibliography}{7}
\bibitem{ASS}
P.~Ara\'ujo, G.~Santos, and M.~Stein,
Semidegree threshold for spanning trees in oriented graphs,
\href{https://arxiv.org/abs/2603.10951v2}{arXiv:2603.10951v2}, 2026.


\bibitem{BGP}
A.~Bonato, P.~Gordinowicz, and P.~Pra{\l}at,
Bounds and constructions for $n$-e.c. tournaments,
\emph{Contrib. Discrete Math.} \textbf{5} (2010), no.~2, 52--66.
doi:\doi{10.55016/ojs/cdm.v5i2.62061}.


\bibitem{Grunbaum}
B.~Gr\"unbaum,
Antidirected Hamiltonian paths in tournaments,
\emph{J. Combin. Theory Ser. B} \textbf{11} (1971), 249--257.
doi:\doi{10.1016/0095-8956(71)90035-9}.



\bibitem{KSS}
J.~Koml\'os, G.~N.~S\'ark\"ozy, and E.~Szemer\'edi,
Spanning trees in dense graphs,
\emph{Combin. Probab. Comput.} \textbf{10} (2001), no.~5, 397--416.
doi:\doi{10.1017/S0963548301004849}.


\bibitem{SteinSurvey}
M.~Stein,
Oriented trees and paths in digraphs,
in F.~Fischer and R.~Johnson (eds.), \emph{Surveys in Combinatorics 2024},
London Math. Soc. Lecture Note Ser., vol.~493, Cambridge University Press,
2024, 271--295.
doi:\doi{10.1017/9781009490559.010}.

\bibitem{SteinTrujillo}
M.~Stein and A.~Trujillo-Negrete,
Antidirected trees in dense digraphs,
\emph{SIAM J. Discrete Math.} \textbf{39} (2025), no.~2, 698--727.
doi:\doi{10.1137/24M165497X}.

\bibitem{SteinZarate}
M.~Stein and C.~Z\'arate-Guer\'en,
Antidirected subgraphs of oriented graphs,
\emph{Combin. Probab. Comput.} \textbf{33} (2024), no.~4, 446--466.
doi:\doi{10.1017/S0963548324000038}.
\end{thebibliography}
\end{document}